 \documentclass[final,1p,times]{elsarticle}
 \usepackage{graphicx}
\usepackage{makeidx}
\usepackage{amssymb}
 \usepackage{amsthm}
\usepackage{amsmath,amssymb,amsopn,amsfonts,mathrsfs,amsbsy,amscd}

\newcommand{\prs}{\langle\;,\;\rangle}

\newcommand{\too}{\longrightarrow}

\newcommand{\esp}{\quad\mbox{and}\quad}
\newcommand{\Ric}{{\mathrm{Ric}}}

\newcommand{\G}{{\mathfrak{g}}}

\newcommand{\ad}{{\mathrm{ad}}}
\newcommand{\tr}{{\mathrm{tr}}}
\newcommand{\ric}{{\mathrm{ric}}}

\newcommand{\B}{{\cal B}}

\newcommand{\al}{\alpha}
\newcommand{\be}{\beta}

\newcommand{\e}{\epsilon}

\newcommand{\la}{\lambda}

\font\bb=msbm10

\def\B{\hbox{\bb B}}
\def\R{\hbox{\bb R}}

\def\N{\hbox{\bb N}}

\newtheorem{theorem}{Theorem}
\newtheorem{lemma}{Lemma}

\newtheorem{corollary}{Corollary}
\newtheorem{definition}{Definition}

\newtheorem{remark}{Remark}
\newtheorem{example}{Example}
\newtheorem{problem}{Problem}
\newtheorem{conjecture}{Conjecture}
\begin{document}

\begin{frontmatter}


 %\title{Left invariant para-K\"ahler and hyper-para-K\"ahler structures on Lie groups\tnoteref{label1}}
 %\author{\corref{cor1}\fnref{label2}}
 
 %\fntext[label3]{This research was conducted within the framework of Action concert\'ee CNRST-CNRS Project SPM04/13.}

\title{ The signature of the Ricci curvature of left-invariant  Riemannian metrics on  nilpotent Lie groups}

 \author[label1,label2,label3]{M.B. Djiadeu Ngaha,  M. Boucetta,  J. Wouafo Kamga}
 \address[label1]{Universit\'{e} de Yaounde I, Facult\'{e} des Sciences, BP 812 Yaound\'{e}, Cameroun\\e-mail: djiadeu@yahoo.fr \\
 }
 \address[label2]{Universit\'e Cadi-Ayyad\\
 Facult\'e des sciences et techniques\\
 BP 549 Marrakech Maroc\\e-mail: m.boucetta@uca.ma
 }
 \address[label3]{Universit\'{e} de Yaounde I, Facult\'{e} des Sciences, BP 812 Yaound\'{e}, Cameroun\\e-mail: wouafoka@yahoo.fr
 
  }

%\author{}

%\address{}

\begin{abstract} Let $(G,h)$ be a nilpotent Lie group endowed with a left invariant Riemannian metric,  $\G$ its Euclidean Lie algebra and $Z(\G)$ the center of $\G$. By using an orthonormal basis adapted to the splitting
 $\G=(Z(\G)\cap[\G,\G])\oplus O^+\oplus (Z(\G)\cap[\G,\G]^\perp)\oplus 
 O^-$, where $O^+$ (resp. $O^-$) is the orthogonal of $Z(\G)\cap[\G,\G]$ in $[\G,\G]$ (resp. is the orthogonal of $Z(\G)\cap[\G,\G]^\perp$ in $[\G,\G]^\perp$), we show that the signature of the Ricci operator of $(G,h)$ is determined by the dimensions of the vector spaces $Z(\G)\cap[\G,\G],$ $Z(\G)\cap[\G,\G]^\perp$ and the signature of a symmetric matrix of order $\dim[\G,\G]-\dim(Z(\G)\cap[\G,\G])$. This permits to associate to $G$ a subset $\mathbf{Sign}(\G)$ of $\N^3$ depending only on the Lie algebra structure, easy to compute and such that, for any left invariant Riemannian metric on $G$, the signature of its Ricci operator belongs to $\mathbf{Sign}(\G)$. We show also that for any nilpotent Lie group of dimension less or equal to 6, $\mathbf{Sign}(\G)$ is actually the set of signatures of the Ricci operators of all left invariant Riemannian metrics on $G$. We give also some general results which support the conjecture that the last result is true in any dimension.

\end{abstract}


\end{frontmatter}

% \linenumbers




%% References


\section{Introduction}\label{section1}
It is a well established fact that there are  deep relations between the topology and the geometry of a manifold on one side, and the curvature of a given Riemannian metric on this manifold on the other side. There is a long list of theorems supporting this fact (see for instance \cite{berge}) and many of them  involve the Ricci curvature. It is a symmetric bilinear tensor and hence has a signature. In the case of a homogeneous Riemannian manifold this signature is the same in any point of the manifold. The determination of the possible signatures of the Ricci operators of $G$-invariant metrics on a $G$-homogeneous space can be useful in many geometrical problems, for instance, in the study of the Ricci flow. This has led naturally to the study of the following open problem that constitutes the topic of this paper.
 \begin{problem}\label{p} For a connected Lie group $G$, determine  all the signatures
of the Ricci operators for all left-invariant Riemannian metrics on $G$.
\end{problem} This problem has been studied mainly in the low dimensions.  In \cite{M} and \cite{LM1, LM2}, Problem \ref{p} has been solved, respectively, in the case of 3-dimensional Lie groups and 4-dimensional Lie groups.
 For Lie groups of
dimension $5$ there are only partial results. In \cite{K5}, A.G. Kremlev, solved Problem \ref{p} in the case of five-dimensional nilpotent Lie groups. In this paper, we study Problem \ref{p} when $G$ is nilpotent. We show that, associated to any nilpotent Lie group $G$, there is a subset $\mathbf{Sign}(\G)$ of $\N^3$ depending only on the Lie algebra $\G$ of $G$, easy to compute and such that, for any left invariant Riemannian metric on $G$, the signature of its Ricci operator belongs to $\mathbf{Sign}(\G)$.
In the case where $\dim G\leq6$, $\mathbf{Sign}(\G)$ is actually the set of signatures of the Ricci operators of all left invariant Riemannian metrics on $G$. We give also some general results which support the conjecture that the last result is true in any dimension.

Now, we  introduce $\mathbf{Sign}(\G)$ and state our main results.
 Throughout this paper,
we will use the following convention. The signature of a symmetric operator $A$ on an Euclidean vector space $V$ is the sequence $(s^-,s^0,s^+)$ where $s^+=\sum_{\la_i>0}\dim \ker(A-\la_i \mathrm{I}_V)$, $s^-=\sum_{\la_i<0}\dim \ker(A-\la_i \mathrm{I}_V)$ and $s^0=\dim\ker A$, where $\la_1,\ldots,\la_r$ are the eigenvalues of $A$.

Let $\G$ be a nilpotent $n$-dimensional Lie algebra, $Z(\G)$ its center and $[\G,\G]$ its derived ideal. Put $d=\dim[\G,\G]$, $k=\dim Z(\G)$ and $\ell=
\dim(Z(\G)\cap[\G,\G])$.
We associate to $\G$ the subset of $\N^3$ \begin{equation}\label{sign}\mathbf{Sign}(\G)=
\left\{(n-d-p+m^-,p+m^0,\ell+m^+):\; \max(k-d,0)\leq p\leq k-\ell,\; m^-+m^0+m^+=d-\ell \right\}.\end{equation}
For instance, if $\G$ is 2-step nilpotent then  $[\G,\G]\subset Z(\G)$ and hence $\mathbf{Sign}(\G)=\left\{(n-
k,k-d,d)    \right\}.$ If $\G$ is a filiform nilpotent Lie algebra then $Z(\G)\subset[\G,\G]$, $\dim Z(\G)=1$,  $\dim[\G,\G]=n-2$ and hence $$\mathbf{Sign}(\G)=\left\{(2+m^-,m^0,1+m^+),\; m^-+m^0+m^+=n-3    \right\}.$$

The signature of the Ricci operator of a left invariant Riemannian metric on Lie group of dimension $n$ belongs to $\{(n^-,n^0,n^+):\; n^-+n^0+n^+=n \}$ whose cardinal is $\frac{(n+1)(n+2)}2$.
 Our first main result reduces drastically the set of possibilities  for a nilpotent Lie group.

\begin{theorem}\label{main} Let $(G,h)$ be a nilpotent Lie group endowed with a left invariant Riemannian metric and $\G$ its Lie algebra. Then the signature of the Ricci operator of $(G,h)$ belongs to $\mathbf{Sign}(\G)$.

\end{theorem}

As an immediate consequence of this result, if $G$ is 2-step nilpotent then any left invariant Riemannian metric on $G$ has the signature of its Ricci operator equal to $(\dim\G-
\dim Z(\G),\dim Z(\G)-\dim[\G,\G],\dim[\G,\G])$. On the other hand, Theorem \ref{main} has the following corollary which gives a new proof to a result proved first in \cite{nikonorov}.

\begin{corollary}\label{coo} Let $(G,h)$ be a noncommutative nilpotent Lie group endowed with a left invariant Riemannian metric and $\G$ its Lie algebra. Then the Ricci operator of $(G,h)$ has at least two negative eigenvalues.

\end{corollary}
Theorem \ref{main} gives a candidate to be the set of  the signatures of the Ricci operators of all left invariant Riemannian metrics on a nilpotent Lie group. Indeed,
our second main result together with Theorem \ref{main} solve  Problem \ref{p} completely for nilpotent Lie groups up to dimension 6. 
\begin{theorem}\label{main0} Let $G$ be a nilpotent Lie group of dimension $\leq6$  and $\G$ its Lie algebra. Then, for any $(s^-,s^0,s^+)\in \mathbf{Sign}(\G)$, there exists a left invariant Riemannian metric on $G$ for which the Ricci operator has signature $(s^-,s^0,s^+)$.

\end{theorem}

Our third main result involves the notion of nice basis.
 Recall that a basis $(X_1,\ldots,X_n)$ of a nilpotent Lie algebra $\G$ is called nice if:
 \begin{enumerate}\item For any $i,j$ with $i\not=j$, $[X_i,X_j]=0$ or there exists  $k$ such that $[X_i,X_j]=C_{ij}^k X_k$ with $C_{ij}^k\not=0$,
 \item If $[X_i,X_j]=C_{ij}^k X_k$ and $[X_s,X_r]=C_{sr}^k X_k$ with $C_{ij}^k\not=0$ and $C_{sr}^k\not=0$ then $\{i,j\}\cap\{s,r\}=\emptyset$.
 
 \end{enumerate} This notion appeared first in \cite{lauret}. 
 One of the most important property of a nice basis $\B$ is that any Euclidean inner product on $\G$ for which $\B$ is orthogonal has its Ricci curvature diagonal in $\B$. The proof of Theorem \ref{main0} is based mainly on the fact that all the nilpotent Lie algebras of dimension less or equal to 6 have a nice basis except one. It is also known (see \cite{lauret1}) that any filiform Lie algebra has a nice basis.

 \begin{theorem}\label{nice} Let $G$ be a nilpotent Lie group such that its Lie algebra $\G$ admits a nice basis and $Z(\G)\subset[\G,\G]$ with $\dim[\G,\G]-\dim Z(\G)=1$. Then, for any $(s^-,s^0,s^+)\in\mathbf{Sign}(\G)$, there exists a left invariant Riemannian metric on $G$ for which the Ricci operator has signature  $(s^-,s^0,s^+)$

 \end{theorem}
 
 This theorem together with Theorem \ref{main} solve Problem \ref{p} for a large class of nilpotent Lie groups. Indeed, in the list of indecomposable seven-dimensional nilpotent Lie algebras given in \cite{ming} there are more than 35 ones satisfying the hypothesis of Theorem \ref{nice}. On the other hand, we will point out the difficulty one can face when trying to generalize Theorem \ref{nice} when $\dim[\G,\G]-\dim Z(\G)\geq2$. We will also give a method using the inverse function theorem to overcome this difficulty. Although, wa have not succeeded yet to show that this method works in the general case, we will use it successfully in the proof of Theorem \ref{main0}. We will refer to this method as inverse function theorem trick. 

The results above,    the tools we will use to establish them and the examples we will give  support the following conjecture.
\begin{conjecture}\label{conj} Let $G$ be a nilpotent Lie group   and $\G$ its Lie algebra. Then, for any $(s^-,s^0,s^+)\in \mathbf{Sign}(\G)$, there exists a left invariant Riemannian metric on $G$ for which the Ricci operator has signature $(s^-,s^0,s^+)$.

\end{conjecture}

The paper is organized as follows. In Sections \ref{section2}-\ref{section3}, we prove a key lemma (see Lemma \ref{theo1}) which implies that, for any nilpotent Lie group $(G,h)$ endowed with a left invariant Riemannian metric, by using an orthonormal basis adapted to the splitting of an Euclidean Lie algebra
 $$\G=(Z(\G)\cap[\G,\G])\oplus O^+\oplus (Z(\G)\cap[\G,\G]^\perp)\oplus 
 O^-,$$ where $O^+$ (resp. $O^-$) is the orthogonal of $Z(\G)\cap[\G,\G]$ in $[\G,\G]$ (resp. is the orthogonal of $Z(\G)\cap[\G,\G]^\perp$ in $[\G,\G]^\perp$),  the signature of the Ricci operator of $(G,h)$ is determined by the dimensions of the vector spaces $Z(\G)\cap[\G,\G],$ $Z(\G)\cap[\G,\G]^\perp$ and the signature of a symmetric matrix of order $\dim[\G,\G]-\dim(Z(\G)\cap[\G,\G])$.  Thereafter, we give a proof of Theorem \ref{main}, its corollary and Theorem \ref{nice}. At the end of Section \ref{section3}, we outline   the inverse function theorem trick that we will use in the proof of Theorem \ref{main0}. Section \ref{section3} is devoted to a proof of Theorem \ref{main0}.
We summarize at the end of the paper in a table all the realizable signatures of Ricci operators on nilpotent Lie groups up to dimension 6. It reduces, according to Theorems \ref{main} and \ref{main0}, to computing $\mathbf{Sign}(\G)$ for any nilpotent Lie algebra of dimension less or equal to 6.
Since we will use  the classification of 5-dimensional and 6-dimensional Lie nilpotent algebras given by Willem A. de Graaf in \cite{graaf}, we give here the lists of  these Lie algebras from \cite{graaf}.

\begin{center}
 \begin{tabular}{|*{2}{c|}}
  \hline
  Lie algebra $\G$ & Nonzero commutators \\
  \hline
  $L_{6,2}= L_{5,2}\oplus \mathbb{R}$& $[e_{1},e_{2}]= e_{3}$\\
  \hline
  $ L_{6,3}= L_{5,3}\oplus \mathbb{R}$& $ [e_{1},e_{2}]= e_{3}, [e_{1},e_{3}]= e_{4}$\\
  \hline
 $ L_{6,4}= L_{5,4}\oplus \mathbb{R}$ & $[e_{1},e_{2}]= e_{5}, [e_{3},e_{4}] = e_{5}$ \\
  \hline
  $L_{6,5}= L_{5,5}\oplus \mathbb{R}$ & $[e_{1},e_{2}]= e_{3}, [e_{1},e_{3}] = e_{5},[e_{2},e_{4}]=e_{5}$\\
  \hline
  $L_{6,6}= L_{5,6}\oplus \mathbb{R}$ & $[e_{1},e_{2}]= e_{3}, [e_{1},e_{3}] = e_{4},[e_{1},e_{4}] = e_{5}, [e_{2},e_{3}] = e_{5}$\\
  \hline
$L_{6,7}= L_{5,7}\oplus \mathbb{R}$ & $[e_{1},e_{2}]= e_{3}, [e_{1},e_{3}] = e_{4}, [e_{1},e_{4}] = e_{5}$\\
  \hline
$L_{6,8}= L_{5,8}\oplus \mathbb{R}$ & $[e_{1},e_{2}]= e_{4}, [e_{1},e_{3}] = e_{5}$ \\
  \hline
  $L_{6,9}= L_{5,9}\oplus \mathbb{R}$ & $[e_{1},e_{2}]= e_{3}, [e_{1},e_{3}] = e_{4},[e_{2},e_{3}] = e_{5}$\\
  \hline
  $L_{6,10}$&$[e_{1},e_{2}]= e_{3},[e_{1},e_{3}]= e_{6},[e_{4},e_{5}]=
  e_{6}$\\
  \hline
  $L_{6,11}$&$[e_{1},e_{2}]= e_{3},[e_{1},e_{3}]= e_{4},[e_{1},e_{4}]=
  e_{6},[e_{2},e_{3}]= e_{6}$,$[e_{2},e_{5}]= e_{6}$\\
  \hline
  $L_{6,12}$&$[e_{1},e_{2}]= e_{3},[e_{1},e_{3}]= e_{4},[e_{1},e_{4}]=
  e_{6},[e_{2},e_{5}]= e_{6}$\\
  \hline
$L_{6,13}$&$[e_{1},e_{2}]= e_{3},[e_{1},e_{3}]= e_{5},[e_{2},e_{4}]=
  e_{5},[e_{1},e_{5}]= e_{6}$, $[e_{3},e_{4}]= e_{6}$\\
  \hline
  $L_{6,14}$&$[e_{1},e_{2}]= e_{3},[e_{1},e_{3}]= e_{4},[e_{1},e_{4}]=
  e_{5},[e_{2},e_{3}]= e_{5}$,\\&$[e_{2},e_{5}]= e_{6},[e_{3},e_{4}]=-e_{6}$\\
  \hline
  $L_{6,15}$&$[e_{1},e_{2}]= e_{3},[e_{1},e_{3}]= e_{4},[e_{1},e_{4}]=
  e_{5},[e_{2},e_{3}]= e_{5}$, $[e_{2},e_{4}]= e_{6}$\\&$[e_1,e_5] = e_6$\\
  \hline
  $L_{6,16}$&$[e_{1},e_{2}]= e_{3},[e_{1},e_{3}]= e_{4},[e_{1},e_{4}]=
  e_{5},[e_{2},e_{5}]= e_{6}$, $[e_{3},e_{4}]=- e_{6}$\\
 \hline
 
  $L_{6,17}$&$[e_{1},e_{2}]= e_{3},[e_{1},e_{3}]= e_{4},[e_{1},e_{4}]=
  e_{5},[e_{1},e_{5}]= e_{6}$, $[e_{2},e_{3}]= e_{6}$\\
   \hline
  $L_{6,18}$&$[e_{1},e_{2}]= e_{3},[e_{1},e_{3}]= e_{4},[e_{1},e_{4}]=
  e_{5},[e_{1},e_{5}]= e_{6}$\\
  \hline
  $L_{6,19}(\epsilon)$&$[e_{1},e_{2}]= e_{4},[e_{1},e_{3}]= e_{5},[e_{2},e_{4}]=
  e_{6},[e_{3},e_{5}]= \epsilon e_{6}$\\
  \hline
  $L_{6,20}$&$[e_{1},e_{2}]= e_{4},[e_{1},e_{3}]= e_{5},[e_{1},e_{5}]=
  e_{6},[e_{2},e_{4}]= e_{6}$\\
  \hline
  $L_{6,21}(\epsilon)$&$[e_{1},e_{2}]= e_{3},[e_{1},e_{3}]= e_{4},[e_{2},e_{3}]=
  e_{5},[e_{1},e_{4}]= e_{6}$, $[e_{2},e_{5}]= \epsilon e_{6}$\\
  \hline
  $L_{6,22}(\epsilon)$&$[e_{1},e_{2}]= e_{5},[e_{1},e_{3}]= e_{6},[e_{2},e_{4}]=\epsilon e_{6},[e_{3},e_{4}]= e_{5}$\\
  \hline
  $L_{6,23}$&$[e_{1},e_{2}]= e_{3},[e_{1},e_{3}]= e_{5},[e_{1},e_{4}]=
  e_{6},[e_{2},e_{4}]= e_{5}$\\
  \hline
  $L_{6,24}(\epsilon)$&$[e_{1},e_{2}]= e_{3},[e_{1},e_{3}]= e_{5},[e_{1},e_{4}]=\epsilon
  e_{6},[e_{2},e_{3}]= e_{6}$, $[e_{2},e_{4}]=e_{5}$\\
  \hline
  $L_{6,25}$&$[e_{1},e_{2}]= e_{3},[e_{1},e_{3}]= e_{5},[e_{1},e_{4}]=
  e_{6}$\\
  \hline
  $L_{6,26}$&$[e_{1},e_{2}]= e_{4},[e_{1},e_{3}]= e_{5},[e_{2},e_{3}]= e_{6}$\\
  \hline
  \end{tabular}
  \end{center}
  \begin{center}
\textbf{Table $1:\epsilon\in\{-1,0,1\}$:} List of six-dimensional nilpotent Lie algebras
\end{center}

\begin{center}
 \begin{tabular}{|*{2}{c|}}
  \hline
  Lie algebra $\G$ & Nonzero commutators \\
  \hline
 $L_{5,2}=L_{3,2}\oplus\R^2$& $[e_{1},e_{2}]= e_{3}$\\
  \hline
  $ L_{5,3}=L_{4,3}\oplus\R$& $ [e_{1},e_{2}]= e_{3}, [e_{1},e_{3}]= e_{4}$\\
  \hline
 $ L_{5,4}$ & $[e_{1},e_{2}]= e_{5}, [e_{3},e_{4}] = e_{5}$ \\
  \hline
  $L_{5,5}$ & $[e_{1},e_{2}]= e_{3}, [e_{1},e_{3}] = e_{5},[e_{2},e_{4}]=e_{5}$\\
  \hline
  $L_{5,6}$ & $[e_{1},e_{2}]= e_{3}, [e_{1},e_{3}] = e_{4},[e_{1},e_{4}] = e_{5},[e_{2},e_{3}] = e_{5} $\\
  \hline
$L_{5,7}$ & $[e_{1},e_{2}]= e_{3}, [e_{1},e_{3}] = e_{4}, [e_{1},e_{4}] = e_{5}$\\
  \hline
$L_{5,8}$ & $[e_{1},e_{2}]= e_{4}, [e_{1},e_{3}] = e_{5}$ \\
  \hline
  $L_{5,9}$ & $[e_{1},e_{2}]= e_{3}, [e_{1},e_{3}] = e_{4},[e_{2},e_{3}] = e_{5}$\\
  \hline
\end{tabular}
\end{center}
\begin{center}
\textbf{Table $2$}: List of five-dimensional nilpotent Lie algebras
\end{center}

\section{ Reduction of the Ricci operator
of a Riemannian Lie group and Ricci signature underestimate   }\label{section2}

In this section, we  prove a key lemma that will play a crucial  role in the proofs of our main results.

A Lie group $G$ together with a left-invariant Riemannian metric $h$ is called a 
\emph{Riemannian Lie group}. The  metric $h$ 
defines a  symmetric positive definite  inner
product $\prs=h(e)$ on the Lie algebra $\G$ of $G$, and conversely, any  symmetric definite positive
inner product on $\G$
gives rise
to an unique  left-invariant Riemannian metric on $G$.\\ We will refer to a Lie
algebra endowed with a  symmetric positive definite  inner
product as an \emph{Euclidean Lie algebra}. \\ The
Levi-Civita connection of $(G,h)$ defines a product $\mathrm{L}:\G\times\G\too\G$ called \emph{ Levi-Civita product}  given by  Koszul's
formula
\begin{eqnarray}\label{levicivita}2\langle
\mathrm{L}_uv,w\rangle&=&\langle[u,v],w\rangle+\langle[w,u],v\rangle+
\langle[w,v],u\rangle.\end{eqnarray}
For any $u,v\in\G$, $\mathrm{L}_{u}:\G\too\G$ is skew-symmetric and $[u,v]=\mathrm{L}_{u}v-\mathrm{L}_{v}u$.
The curvature on $\G$ is given by
$$
 \label{curvature}K(u,v)=\mathrm{L}_{[u,v]}-[\mathrm{L}_{u},\mathrm{L}_{v}].
$$
The Ricci curvature on $\G$ is defined by $\mathrm{ric}(u,v)=\mathrm{tr}\left(w\too
K(u,w)v\right)$. 
The \emph{mean curvature vector} on $\G$ is the vector $H$ defined by the following relation
$\label{mean}\langle H,u\rangle=\mathrm{tr(\ad_u)},
$ where $\ad_u:\G\too\G$, $v\mapsto [u,v]$.
It is well-known that $\mathrm{ric}$ is  given by
\begin{equation}\label{ricci1}
 \ric(u,v)=-\frac12\tr(\ad_u\circ\ad_v)-\frac12\tr(\ad_u\circ
\ad_v^*)-\frac14\tr(J_u\circ J_v)-\frac12{\langle}\ad_Hu,v{\rangle}-
\frac12{\langle}\ad_Hv,u{\rangle},\end{equation}
where $J_u$ is the
skew-adjoint endomorphism given by $J_uv=\ad_v^*u$ ($\ad_u^*$ is the adjoint of $\ad_u$
with respect to $\prs$). The Ricci operator is
 the auto-adjoint endomorphism $\mathrm{Ric}:\G\too\G$ given by $\langle \mathrm{Ric}(u),v\rangle=\ric(u,v)$. The signature of $\Ric$ is called \emph{ Ricci signature} of $(G,h)$ or $(\G,\prs)$.

 We consider now the Lie subalgebra of left invariant Killing vector fields on $G$ given by
 $$K(\prs)=\left\{u\in\G,\ad_u+\ad_u^*=0\right\}.$$ It contains obviously the center $Z(\G)$ of $\G$. Put $K^+(\prs)=K(\prs)\cap[\G,\G]$ and $K^-(\prs)=K(\prs)\cap[\G,\G]^\perp$.  
 Denote by $O^+$ (resp. $O^-$) the orthogonal of $K^+(\prs)$ in $[\G,\G]$ (resp. the orthogonal of $K^-(\prs)$ in $[\G,\G]^\perp$).
 Then
 \begin{equation}\label{split}\G=K^+(\prs)\oplus O^+\oplus  K^-(\prs)\oplus O^-.\end{equation}
 We call this splitting \emph{characteristic splitting} of $(\G,\prs)$ and any  basis of $\G$ of the form $\B_1\cup \B_2\cup \B_3\cup\B_4$ where $\B_1$, $\B_2$, $\B_3$ and $\B_4$ are, respectively, bases of $K^+(\prs), O^+,  K^-(\prs), O^-$ is called \emph{characteristic basis}. 
 
 \begin{lemma}\label{theo1}With the hypothesis and the notations above, let $n_1=\dim K^+(\prs)$, $n_2=\dim O^+$, $n_3=\dim K^-(\prs)$ and $n_4=\dim O^-$. Then  we have: \begin{enumerate}\item[$(i)$] $K^-(\prs)\subset\ker\Ric$ and if $K^+(\prs)\not=\{0 \}$ then the restriction of $\ric$ to 
 $K^+(\prs)$ is positive definite. \item[$(ii)$] If 
 $O^-\not=\{0 \}$ then the restriction of $\ric$ to 
  $O^-$ is negative definite and $\ric(K^+(\prs),O^-)=0$.
 \item[$(iii)$] For any  characteristic basis $\B$ of $\G$, the matrix of the Ricci tensor in $\B$ is given by  
 \[ \mathrm{Mat}(\ric,\B)= \frac12\left[
 \begin{array}{cccc}
 Z & V & 0 & 0 \\
       V^{t} & X & 0 & W \\
       0 & 0 & 0 & 0 \\
       0 & W^{t} & 0 & Y
     \end{array}
   \right] ,\]where $Z,X,Y$ are square matrices of order $n_1$, $n_2$, $n_4$, respectively, and   the Ricci signature of $(\G,\prs)$ is given by
  \begin{equation}\label{eqimportant}(s^-,s^0,s^+)=(\dim[\G,\G]^\perp-\dim K^-(\prs)+m^-,\dim K^-(\prs)+m^0,\dim K^+(\prs)+m^+),\end{equation}
 where $(m^-,m^0,m^+)$ is the signature of the symmetric matrix
 \begin{equation}\label{matrix}\mathrm{R}(\ric,\B)=X-V^tZ^{-1}V-WY^{-1}W^t.\end{equation}
 \end{enumerate}
 
 \end{lemma}
 \begin{proof} First remark that, for any $u\in\G$, $J_u$ is skew-symmetric and $J_u=0$ iff $u\in[\G,\G]^\perp$. With this remark in mind,
 by using \eqref{ricci1}, we get  for any $u\in K^+(\prs)$, $\ric(u,u)=-\frac14\tr(J_u^2)\geq0$ and $\ric(u,u)=0$ if and only if $J_u=0$. This shows that the restriction of $\ric$ to $K^+(\prs)$ is definite positive. On the other hand, for any $u\in O^-$, by using \eqref{ricci1} we get $\ric(u,u)=-\frac14\tr((\ad_u+\ad_u^*)^2)\leq0$ and $\ric(u,u)=0$ iff $u\in K(\prs)$. 
 This shows that the restriction of $\ric$ to $O^-$ is negative definite. We have also, for any $u\in K^-(\prs)$ and any $v\in\G$, $\ric(u,v)=0$. Finally, for any $u\in K^+(\prs)$ and any $v\in O^-$, $\ric(u,v)=0$ this completes the proof of $(i)-(ii)$.
 
  In any characteristic basis $\B$ of $\G$, according to the results shown in $(i)-(ii)$, the matrix $\mathrm{R}(\ric,\B)$ has the desired form. Put \[ Q =  \left[
                \begin{array}{cccc}
                  I_{n_{1}} & -Z^{-1}V & 0 & 0 \\
                  0 & I_{n_{2}} & 0 & 0 \\
                  0 & 0 & I_{n_{3}} & 0 \\
                  0 & -Y^{-1}W^{t} & 0 & I_{n_{4}}
                \end{array}
              \right].\]         
 We can check easily that   \[ Q^{t}\;\mathrm{Mat}(\ric,\B)\; Q  = \frac{1}{2}\left[
                              \begin{array}{cccc}
                                Z & 0 & 0 & 0 \\
                                0& \mathrm{R}(\ric,\B) & 0 & 0 \\
                                0 & 0 & 0 & 0 \\
                                0 & 0 & 0 & Y
                              \end{array}
                            \right].\]
 This formula combined with the results in $(i)-(ii)$ give the desired formula for the signature of $\ric$.\qedhere

 \end{proof}
 
 \begin{definition}\label{def} Let $(G,h)$ be a Riemannian Lie group and $(\G,\prs)$ its associated Euclidean Lie algebra.
 \begin{enumerate}\item[$\bullet$] We call $(r^-,r^0,r^+)=(\dim[\G,\G]^\perp-
 \dim K^-(\prs),\dim K^-(\prs),\dim K^+(\prs))$ the Ricci signature underestimate of $(\G,\prs)$.
 \item[$\bullet$] For any characteristic basis $\B$ of $\G$, we call $\mathrm{R}(\ric,\B)$ defined by \eqref{matrix} \emph{ reduced matrix } of the Ricci curvature in $\B$. It is a symmetric $(s\times s)$-matrix with   $s=\dim[\G,\G]-\dim K^+(\prs)$.
 
 \end{enumerate}
 
 \end{definition}

  Note that the order of  $\mathrm{R}(\ric,\B)$  is zero iff $[\G,\G]\subset K(\prs)$. In this case $K(\prs)=[\G,\G]\oplus K^-(\prs)$ and  we get:
 \begin{corollary}\label{co1} Let $(G,h)$ be a Riemannian Lie group such that $[\G,\G]\subset K(\prs)$. Then the signature of the Ricci curvature of $h$ is given by
 $$(s^-,s^0,s^+)=(\dim\G-\dim K(\prs),\dim K(\prs)-\dim[\G,\G],\dim [\G,\G]).$$
 \end{corollary}
 
 \begin{remark} The case where the Riemannian metric is bi-invariant $($$\G=K(\prs)$$)$ is a particular case of the situation in Corollary \ref{co1} and in this case $Z(\G)=[\G,\G]^\perp$ and hence the signature is given by
 \[ (s^-,s^0,s^+)=(0,\dim Z(\G),\dim [\G,\G]). \]

 \end{remark}
 
 \section{ Ricci signature underestimates  in nilpotent Riemannian Lie groups and a proof of Theorems \ref{main} and \ref{nice}}\label{section3}
 
 In this section, we will show that Lemma \ref{theo1} turn out to be very useful in the case of nilpotent Riemannian Lie groups that permits us to prove Theorems \ref{main} and \ref{nice}.
 
 \subsection{Preliminaries}
 
 Let $(G,h)$ be  a nilpotent Riemannian Lie group. The  formula \eqref{ricci1} becomes in this case  quite simple
 \begin{equation*}
 \ric(u,v)=-\frac12\tr(\ad_u\circ\ad_v^*)-\frac14\tr(J_u\circ J_v)
 =-\frac12\langle\ad_u,\ad_v\rangle_1+\frac14\langle J_u,J_v\rangle_1,
 \end{equation*}where $\prs_1$ is the Euclidean product on $\mathrm{End}(\G)$ associated to $\prs$. In particular, if $(e_1,\ldots,e_n)$ is an orthonormal 
 basis of $\G$ then
 \begin{equation}\label{riccinilpotent}
 \ric(u,v)=-\frac12\sum_{i,j}\langle[u,e_i],e_j\rangle
 \langle[v,e_i],e_j\rangle+\frac12\sum_{i<j}\langle[e_i,e_j],u\rangle
 \langle[e_i,e_j],v\rangle.\end{equation}
 Moreover, since a skew-symmetric nilpotent endomorphism must vanishes then  $K(\prs)=Z(\G)$. This simple fact combined with the result of Lemma \ref{theo1} will have surprising consequences. Note first that, as a particular case of Corollary \ref{co1}, we get the following result which first appeared in \cite{bouc} and which solves Problem 1 for 2-step nilpotent Lie groups.
 \begin{corollary}\label{co2} Let $G$ be a 2-step nilpotent Lie group. Then, for any left-invariant Riemannian metric on $G$, the signature of its Ricci curvature is given by
 $$(s^-,s^0,s^+)=(\dim\G-\dim Z(\G),\dim Z(\G)-\dim[\G,\G],\dim [\G,\G]).$$
 \end{corollary}
 \subsection{Proof  of Theorem \ref{main} and Corollary \ref{coo}}
 
 \begin{proof}
 Let $(G,h)$ be a nilpotent Riemannian Lie group. We distinguish two cases.
 \begin{enumerate}\item[$\bullet$] $Z(\G)\subset[\G,\G]$. In this case, it is obvious that the Ricci signature underestimate  of $(\G,\prs)$ is given by
 \[ (r^-,r^0,r^+)=(\dim\G-\dim[\G,\G],0,\dim Z(\G)). \]
 On the other hand, by using \eqref{sign}, one can see easily that
 \[ \mathbf{Sign}(\G)=\left\{(r^-+m^-,r^0+m^0,r^++m^+),\; m^-+m^0+m^+=\dim[\G,\G]-\dim Z(\G)  \right\}. \]
 According to Lemma \ref{theo1}, the Ricci signature of $h$ belongs to $\mathrm{Sign}(\G)$ and we obtain the result in this case.  Corollary \ref{coo} follows from the fact that $r^-=\dim\G-\dim[\G,\G]\geq2$. In a nilpotent Lie algebra the derived ideal is always of codimension greater than 2.
 \item[$\bullet$]  $Z(\G)\nsubset[\G,\G]$. Choose a complement $I$ of $Z(\G)\cap[\G,\G]$ in $Z(\G)$ and a complement $U$ of $[\G,\G]\oplus I$ in $\G$. Thus $\G=\G_1\oplus I$ where $\G_1=[\G,\G]\oplus U$ is an ideal of $\G$ and $I$ is a central ideal. Moreover, $Z(\G_1)=Z(\G)\cap[\G,\G]$ and $[\G,\G]=[\G_1,\G_1]$. By using the same notations as in \eqref{sign}, we get that  the Ricci signature underestimate  of $(\G,\prs)$ is given by
 \[ (r^-,r^0,r^+)=(n-d-p,p,\ell),\quad p= \dim (Z(\G)\cap[\G,\G]^\perp).\]
 We have obviously $ p\leq\dim I=\dim Z(\G)-\dim(Z(\G)\cap[\G,\G])$ and
 \[ p=\dim Z(\G)+\dim\G-\dim[\G,\G]-\dim(Z(\G)+[\G,\G]^\perp)\geq \dim Z(\G) -\dim[\G,\G].  \]According to Lemma \ref{theo1}, the Ricci signature of $h$ belongs to $\mathbf{Sign}(\G)$ and we obtain the result in this case. Corollary \ref{coo} follows from the fact that 
 \begin{eqnarray*}r^-&=&\dim\G-\dim[\G,\G]-\dim (Z(\G)\cap[\G,\G]^\perp)\\
 &=&
 \dim\G_1-\dim[\G_1,\G_1]+\dim I-\dim (Z(\G)\cap[\G,\G]^\perp)\\
 &=&\dim\G_1-\dim[\G_1,\G_1]+\dim Z(\G)-\dim (Z(\G)\cap[\G,\G])-\dim (Z(\G)\cap[\G,\G]^\perp)\\&\geq&\dim\G_1-\dim[\G_1,\G_1] \geq2.\end{eqnarray*}
 
 \end{enumerate}

 \end{proof}

 \subsection{Proof of Theorem \ref{nice}}
 
 \begin{proof} We have obviously $\mathbf{Sign}(\G)=\{(n-d+m^-,m^0,d-1+m^+):\; m^-+m^0+m^+=1\}$, where $d=\dim[\G,\G]$ and $n=\dim\G$.
  Note first that we can choose a nice basis
 $\B=(X_i)_{i=1}^n$ where $Z(\G)=
 \mathrm{span}\{X_i \}_{i=1}^{d-1}$ and $[\G,\G]=
 \mathrm{span}\{X_i \}_{i=1}^{d}$. Indeed, suppose that $\B=(X_i)_{i=1}^n$ with $[\G,\G]=
 \mathrm{span}\{X_i \}_{i=1}^d$. Let $z=\sum_{i=1}^da_iX_i\in Z(\G)$.  Suppose that there exists $a_i\not=0$ and $X_i\notin Z(\G)$. Then there exists $\ell\in\{1,\ldots,n\}$ such that $[X_\ell,X_i]\not=0$. So we get
 $\sum_{j=1}^da_j[X_\ell,X_j]=0.$
 From the properties of a nice basis we deduce that $\{[X_\ell,X_j],j=1,\ldots,d,\;[X_\ell,X_j]\not=0 \}$ is a linearly independent family and hence $a_i=0$. This shows that $\{X_i,\;X_i\in Z(\G)\}$ is basis of $Z(\G)$.

 We consider the Euclidean product $\prs$ on $\G$ for which $\B$ is orthogonal and $a_i=\langle X_i,X_i\rangle$. It is obvious that $\B$ is a characteristic basis of $(\G,\prs)$ and it is also nice of $\G$ so $\mathrm{M}(\ric,\B)$ is diagonal and hence $\mathrm{R}(\ric,\B)$ is also diagonal. According to Lemma \ref{theo1}, the reduced matrix has order 1 and is given by
 $\mathrm{R}(\ric,\B)=(2\ric(X_d,X_d))$. Moreover, the Ricci signature of $(\G,\prs)$ is given by $(n-d+m^-,m^0,d-1+m^+)$ where $(m^-,m^0,m^+)$ is the signature of $\mathrm{R}(\ric,\B)$. To complete the proof, we will show that we can choose suitable $a_i$ so that $\ric(X_d,X_d)$ can be either zero, positive or negative.

 Denote by $C_{ij}^k$ the structure constants of the Lie bracket in $\B$. The basis $(e_i)_{i=1}^n=(\frac1{\sqrt{a_i}}X_i)_{i=1}^n$ is an orthonormal basis of $\G$ and from \eqref{riccinilpotent}
 \begin{eqnarray} 2\ric(X_d,X_d)
  &=& \sum_{i<j}\frac{(C_{ij}^d)^2 a_d^2 }{a_ia_j}-\sum_{i,j}
  \frac{(C_{id}^{j})^2 a_{j} }{a_i}\label{eqricci}.\end{eqnarray}
   Note that for any $(i,j)$, such that $[X_i,X_j]=C_{ij}^dX_d$ with $C_{ij}^d\not=0$,
  $i\not=d$ and $j\not=d$. Indeed, if $i=d$, we have $[X_d,X_j]=C_{d j}^d X_d$ and hence $X_d$ is an eigenvector of $\ad_{X_j}$ with the real non zero eigenvalue $-C_{d i}^d$ which is impossible since $\ad_{X_j}$ is nilpotent. We have also that if $[X_d,X_i]=C_{di}^jX_j$ with $C_{di}^j\not=0$ then $i\not=d$ and $j\not=d$.
  So
  \[ \ric(X_d,X_d)=\al a_d^2-\be. \]Now since $X_d\in[\G,\G]\setminus Z(\G)$,
   $\al>0$, $\be>0$ and both $\al$ and $\be$ depend only on $a_i$ with $i\not=d$. So we can choose $a_d$ such that $\ric(X_d,X_d)=0,$ $>0$ or $<0$. This completes the proof.
  \end{proof}  
  
  One can ask naturally if this theorem is still true when $\dim[\G,\G]-\dim Z(\G)\geq2$. By looking to the proof given here, one can conjecture that the answer is true, it suffices to solve some  systems of polynomial equations. This can be very difficult. To be precise, we will point out the difficulty one can face when trying to generalize Theorem \ref{nice} when $\dim[\G,\G]-\dim Z(\G)\geq2$. We will also give a method to overcome this difficulty. Although, we have not succeeded yet to show that this method works in the general case, we will use it successfully in the proof of Theorem \ref{main0}.
  
  \subsection{Inverse function theorem trick} Suppose that $\G$ is a nilpotent Lie algebra having a nice basis $\B$ and satisfying $Z(\G)\subset[\G,\G]$. Write
   $\B=(X_i)_{i=1}^n$ where $(X_i)_{i=1}^\ell$ is a basis of $Z(\G)$ and $(X_i)_{i=1}^d$ is a basis of $[\G,\G]$. We have obviously
    $$\mathbf{Sign}(\G)=\left\{(n-d+m^-,m^0,\ell+m^+):\; m^-+m^0+m^+=d-\ell\right\}.$$
   
   We consider the Euclidean product $\prs$ on $\G$ for which $\B$ is orthogonal and $a_i=\langle X_i,X_i\rangle$. It is clear that $\B$ is a characteristic basis of $(\G,\prs)$ and it is also nice so $\mathrm{M}(\ric,\B)$ is diagonal. According to Lemma \ref{theo1}, the reduced matrix has order $d-\ell$ and is diagonal and given by
   $$\mathrm{R}(\ric,\B)=\mathrm{diag}(2\ric(X_{\ell+1},X_{\ell+1}),\ldots,2\ric(X_d,X_d)).$$ Moreover, the signature is given by $(n-d+m^-,m^0,\ell+m^+)$ where $(m^-,m^0,m^+)$ is the signature of $\mathrm{R}(\ric,\B)$. 
  According to \eqref{eqricci}, for any $i=\ell+1,\ldots,d$, we can write in a unique way
   $$2\ric(X_i,X_i)=\frac{F_{i-\ell}(a_1,\ldots,a_n)}{a_{n_1}\ldots a_{n_i}},$$ where $F_{i-\ell}$ is a homogeneous polynomial on $(a_1,\ldots,a_n)$. So to  generalize Theorem \ref{nice} when $d-\ell\geq2$, it suffices to find suitable values of $(a_1,\ldots,a_n)$ such that $(F_i(a_1,\ldots,a_n))_{i=1}^{d-\ell}$ have all the  possible signs. It is very difficult in the general case. We give now a situation where we can conclude.

    Suppose that there exists $(\al_1,\ldots,\al_n)$ such that $F_j(\al_1,\ldots,\al_n)=0$ for $j=1,\ldots,d-\ell$ and define $$F:
   \{(x_1,\ldots,x_{d-\ell})\in\R^{d-\ell}, x_i>0 \}\too \R^{d-\ell}$$ by
   $$F(x_1,\ldots,x_{d-\ell})=(F_1(\al_1,\ldots,\al_\ell,x_1,\ldots,x_{d-\ell},\al_{d+1},\ldots,\al_n),
   \ldots,F_{d-\ell}(\al_1,\ldots,\al_\ell,x_1,\ldots,x_{d-\ell},\al_{d+1},\ldots,\al_n)).$$
   We have $F(\al_{\ell+1},\ldots,\al_d)=0$ and if the differential $DF(\al_{\ell+1},\ldots,\al_d)$ is invertible we can apply the inverse function theorem  and hence $F$ realizes a diffeomorphism from an open set centred in $(\al_{\ell+1},\ldots,\al_d)$ into an open ball centred in $(0,,\ldots,0)$.
    So, for a suitable choice of $a_i$, $\mathrm{R}(\ric,\B)$ can have all the possible signatures. 
    
    So far  we have shown that Theorem \ref{nice} is  true when $\dim[\G,\G]-\dim Z(\G)\geq2$ if there exists $(\al_1,\ldots,\al_n)$ with $\al_1>0,\ldots,\al_n>0$ satisfying $F_j(\al_1,\ldots,\al_n)=0$ for $j=1,\ldots,d-\ell$ and $\det DF(\al_{\ell+1},\ldots,\al_d)\not=0$. 
   
   \begin{definition} \label{nicedef}We call nice a nilpotent Lie algebra $\G$ with $Z(\G)\subset [\G,\G]$ and having a nice basis for which there exists $(\al_1,\ldots,\al_n)$ with $\al_1>0,\ldots,\al_n>0$ satisfying $F_j(\al_1,\ldots,\al_n)=0$ for $j=1,\ldots,d-\ell$ and $\det DF(\al_{\ell+1},\ldots,\al_d)\not=0$.
   
   \end{definition}
   
   So, according to our study above, we have the following result.
   
   \begin{theorem}\label{nice2} Let $G$ be a nilpotent Lie group such that its Lie algebra $\G$ is nice. Then for any $(s^-,s^0,s^+)\in\mathrm{Sign}(\G)$ there exists a left invariant Riemannian metric on $G$ such that its Ricci signature is $(s^-,s^0,s^+)$.

   \end{theorem}
   
   \begin{remark} It is seems reasonable to conjecture that any nilpotent Lie algebra $\G$ with $Z(\G)\subset [\G,\G]$ and having a nice basis is actually nice.
   
   \end{remark}
   
   We give now two examples of nice nilpotent Lie algebras.
   
   \begin{example}\label{exem1}\begin{enumerate}\item We consider the 7-dimensional nilpotent Lie algebra labelled $(12457L1)$ in \cite{ming}  given by
   \[ [e_1,e_2]=e_3,\;[e_1,e_3]=e_4,\;[e_1,e_4]=-e_6,\;[e_1,e_6]=e_7,\;[e_2,e_3]=e_5,\;[e_2,e_5]=-e_6,\;[e_3,e_5]=-e_7. \]
   We have $Z(\G)=\{e_7\}\subset[\G,\G]=\{e_3,e_4,e_5,e_6,e_7 \}$ and $\B=(e_7,e_3,e_4,e_5,e_6,e_1,e_2)$ is a nice basis. Let compute $2\ric(e_i,e_i)$ for $i=3,\ldots,6$ for the metric for which $\B$ is orthogonal with $\langle e_i,e_i\rangle=a_i$. By applying \eqref{eqricci}, we get
   \begin{eqnarray*}
   2\ric(e_3,e_3)&=&\frac{a_3^2}{a_1a_2}-\frac{a_4}{a_1}-\frac{a_5}{a_2}-\frac{a_7}{a_5}=\frac{a_3^2a_5-a_2a_4a_5-a_1a_5^2-a_1a_2a_7}{a_1a_2a_5}=\frac{F_1(a_1,\ldots,a_7)}{a_1a_2a_5},\\
   2\ric(e_4,e_4)&=&\frac{a_4^2}{a_1a_3}-\frac{a_6}{a_1}=
   \frac{a_4^2-a_3a_6}{a_1a_3}=\frac{F_2(a_1,\ldots,a_7)}{a_1a_3},\\
   2\ric(e_5,e_5)&=&\frac{a_5^2}{a_2a_3}-\frac{a_6}{a_2}-\frac{a_7}{a_3}
   =\frac{a_5^2-a_3a_6-a_2a_7}{a_2a_3}=\frac{F_3(a_1,\ldots,a_7)}{a_2a_3},\\
   2\ric(e_6,e_6)&=&\frac{a_6^2}{a_2a_5}+\frac{a_6^2}{a_1a_4}-\frac{a_7}{a_1}
   =\frac{(a_1a_4+a_2a_5)a_6^2-a_2a_4a_5a_7}{a_1a_2a_4a_5}=\frac{F_4(a_1,\ldots,a_7)}{a_1a_2a_4a_5}.
   \end{eqnarray*}
    The sequence $\left(\frac{7}{240},\frac{1127}{1200},1,1,\frac{7}{5},1,\frac{1152}{1127}\right)$ is a solution of the equations
   $F_i(\al_1,\ldots,\al_7)=0$ for $i=1,\ldots,4$ and satisfies $\det DF(\al_3,\al_4,\al_5,\al_6)\not=0$ and hence this Lie algebra is nice. 
   
   \item We consider the $\N$-graded filiform $n$-dimensional Lie algebra $\mathfrak{m}_0(n)=\mathrm{span}\{X_1,\ldots,X_n\}$ with the non vanishing Lie brackets
   $[X_1,X_i]=X_{i+1}$, $i=2,\ldots,n-1$. We have
   $$\mathbf{Sign}(\mathfrak{m}_0(n))=\left\{ (2+m^-,m^0,1+m^+),\;m^-+m^0+m^+=n-3  \right\}.$$
   Let $\prs$ be  the Euclidean inner product on $\mathfrak{m}_0(n)$ for which $(X_i)_{i=1}^n$ is an orthogonal basis with $\langle X_i,X_i\rangle=a_i$. The basis $\B=(X_n,X_3,\ldots,X_{n-1},X_1,X_2)$ is a characteristic basis of $\prs$ and $\mathrm{R}(\ric,\B)=\mathrm{diag}(2\ric(X_i,X_i))_{i=3}^{n-1}$. 
   By using \eqref{eqricci}, we get for any $i=3,\ldots,n-1$
   $$2\ric(X_i,X_i)=\frac{a_i^2}{a_1a_{i-1}}-\frac{a_{i+1}}{a_1}=\frac{a_i^2-a_{i-1}a_{i+1}}{a_1a_{i-1}}
   =\frac{F_{i-2}(a_1,\ldots,a_n)}{a_1a_{i-1}}.$$
   It is obvious that $F_i(1,\ldots,1)=0$ and $\det DF(1,\ldots,1)\not=0$ and hence 
   $\mathfrak{m}_0(n)$ is nice.
   
   \end{enumerate}
   
   \end{example}

\section{Proof of Theorem \ref{main0}}\label{section4}

\begin{proof} The proof goes as follows. There are, up to an isomorphism, 44 non abelian nilpotent Lie algebras of dimension less or equal to 6: 1 of dimension 3, 2 of dimension 4, 8 of dimension 5 and 33 of dimension 6 (see Tables 1 and 2). Among these Lie algebras, 12 are 2-step nilpotent and we can apply Corollary \ref{co2}, 10 satisfy the hypothesis of Theorem \ref{nice} and 15 are nice in the sense of Definition \ref{nicedef} and we can apply Theorem \ref{nice2}. At the end, we are left with 7 Lie algebras needing each of them a special treatment.

The Lie algebras $L_{3,2}$, $L_{4,2}$, $L_{5,2}$, $L_{5,4}$, $L_{5,8}$, $L_{6,2}$, $L_{6,4}$ $L_{6,8}$, $L_{6,22}(\e)$, $L_{6,26}$ are obviously 2-step nilpotent and we can apply Corollary \ref{co2}.

The Lie algebras $L_{4,3}$, $L_{5,5}$, $L_{5,9}$, $L_{6,10}$, $L_{6,19}(0)$, $L_{6,23}$, $L_{6,24}(\e)$ and $L_{6,25}$ satisfy clearly the hypothesis of Theorem \ref{nice}.

We will show now that  the Lie algebras $L_{5,6}$, $L_{5,7}$, $L_{6,12}$, $L_{6,13}$, $L_{6,14}$, $L_{6,15}$, $L_{6,16}$, $L_{6,17}$, $L_{6,18}$, $L_{6,19}(\e\not=0)$, $L_{6,20}$, $L_{6,21}(0)$ and $L_{6,21}(\e\not=0)$ are nice in the sense of Definition \ref{nicedef} so that we can apply Theorem \ref{nice}.
 Since the computations are straightforward, we will give for any Lie algebra among these Lie algebras, a nice basis $\B$, the reduced matrix in $\B$ associated to the Euclidean product for which $\B$ is diagonal with $a_i=\langle e_i,e_i\rangle$ and the value $(\al_1,\ldots,\al_n)$ appearing in Definition \ref{nicedef}. Note that $\B_0=(e_1,\ldots,e_n)$ is the basis appearing in Tables 1 and 2.

{\footnotesize
\begin{eqnarray*}
L_{5,6}&:&\left[ (e_5,e_3,e_4,e_1,e_2),\mathrm{diag}\left(\frac{a_3^2-a_2a_4-a_1a_5}{a_1a_2},
\frac{a_4^2-a_3a_5}{a_1a_3}\right),\left(\frac12,\frac12,1,1,1\right)  \right].\\
L_{5,7}&:&\left[ (e_5,e_3,e_4,e_1,e_2),\mathrm{diag}\left(\frac{a_3^2-a_2a_4}{a_1a_2},
\frac{a_4^2-a_3a_5}{a_1a_3}\right),\left(1,1,1,1,1\right)  \right].\\
L_{6,12}&:&\left[ (e_6,e_3,e_4,e_1,e_2,e_5),\mathrm{diag}\left(\frac{a_3^2-a_2a_4}{a_1a_2},
\frac{a_4^2-a_3a_6}{a_1a_3}\right),\left(1,1,1,1,1,1\right)  \right].\\
L_{6,13}&:&\left[ (e_6,e_3,e_5,e_1,e_2,e_4),\mathrm{diag}\left(\frac{a_4a_3^2-a_2a_4a_5-a_1a_2a_6}{a_1a_2a_4},
\frac{(a_2a_4+a_1a_3)a_5^2-a_2a_3a_4a_6}{a_1a_2a_3a_4}\right),\left(1,2,2,1,1,1\right)  \right].\\
L_{6,14}&:&\left[ (e_6,e_3,e_4,e_5,e_1,e_2),\mathrm{diag}\left(\frac{a_4a_3^2-a_2a_4^2-a_1a_4a_5-a_1a_2a_6}{a_1a_2a_4},
\frac{a_4^2-a_3a_5-a_1a_6}{a_1a_3},\frac{(a_2a_3+a_1a_4)a_5^2-a_1a_3a_4a_6}{a_1a_2a_3a_4}\right),
\left(\frac{27}{200},\frac{3}{40},1,3,5,\frac{800}{27}\right)  \right].\\
L_{6,15}&:&\left[ (e_6,e_3,e_4,e_5,e_1,e_2),\mathrm{diag}\left(\frac{a_3^2-a_2a_4-a_1a_5}{a_1a_2},
\frac{a_2a_4^2-a_2a_3a_5-a_1a_3a_6}{a_1a_2a_3},\frac{(a_2a_3+a_1a_4)a_5^2-a_2a_3a_4a_6}{a_1a_2a_3a_4}\right),
\left(\frac{4}{3},\frac{4}{3},2,1,1,1\right)  \right].\\
L_{6,16}&:&\left[ (e_6,e_3,e_4,e_5,e_1,e_2),\mathrm{diag}\left(\frac{a_4a_3^2-a_2a_4^2-a_1a_2a_6}{a_1a_2a_4},
\frac{a_4^2-a_3a_5-a_1a_6}{a_1a_2a_3},\frac{a_2a_5^2-a_1a_4a_6}{a_1a_2a_4}\right),
\left(\frac{1}{3},\frac{1}{3},\frac{2}{3},1,1,1\right)  \right].\\
L_{6,17}&:&\left[ (e_6,e_3,e_4,e_5,e_1,e_2),\mathrm{diag}\left(\frac{a_3^2-a_2a_4-a_1a_6}{a_1a_2},
\frac{a_4^2-a_3a_5}{a_1a_3},\frac{a_5^2-a_4a_6}{a_1a_4}\right),
\left(\frac{1}{2},\frac{1}{2},1,1,1,1\right)  \right].\\
L_{6,18}&:&\left[ (e_6,e_3,e_4,e_5,e_1,e_2),\mathrm{diag}\left(\frac{a_3^2-a_2a_4}{a_1a_2},
\frac{a_4^2-a_3a_5}{a_1a_3},\frac{a_5^2-a_4a_6}{a_1a_4}\right),
\left(1,1,1,1,1,1\right)  \right].\\
L_{6,19}(\e\not=0)&:&\left[ (e_6,e_4,e_5,e_1,e_2,e_3),\mathrm{diag}\left(\frac{a_4^2-a_1a_6}{a_1a_2},
\frac{a_5^2-a_1a_6}{a_1a_3}\right),
\left(1,1,1,1,1,1\right)  \right].\\
L_{6,20}&:&\left[ (e_6,e_4,e_5,e_1,e_2,e_3),\mathrm{diag}\left(\frac{a_4^2-a_1a_6}{a_1a_2},
\frac{a_5^2-a_3a_6}{a_1a_3}\right),
\left(1,1,1,1,1,1\right)  \right].\\
L_{6,21}(0)&:&\left[ (e_5,e_6,e_3,e_4,e_1,e_2),\mathrm{diag}\left(\frac{a_3^2-a_2a_4-a_1a_5}{a_1a_2},
\frac{a_4^2-a_3a_6}{a_1a_3}\right),
\left(2,\sqrt{2},2,\sqrt{2},1,1\right)  \right].\\
L_{6,21}(\e\not=0)&:&\left[ (e_6,e_3,e_4,e_5,e_1,e_2),\mathrm{diag}\left(\frac{a_3^2-a_1a_4-a_2a_5}{a_1a_2},
\frac{a_4^2-a_3a_6}{a_1a_3},\frac{a_5^2-a_3a_6}{a_2a_3}\right),
\left(\frac{1}{2},\frac{1}{2},1,1,1,1\right)  \right].\\
\end{eqnarray*}
}

To complete the proof, we treat now the seven remaining Lie algebras using a case by case approach.\\

$\bullet$ {\bf The Lie algebra $L_{6,11}$ }. \\

This is the only Lie algebra in the list which has no nice basis. Its center is contained in its derived ideal. 
We have $\G=L_{6,11}=\mathrm{span}\{e_1,\ldots,e_6\}$ with the non vanishing Lie brackets
$$[e_{1},e_{2}]= e_{3},\; [e_{1},e_{3}] = e_{4},\;[e_{1},e_{4}] = e_{6},[e_{2},e_{3}] = e_{6},\;[e_{2},e_{5}] = e_{6} $$ and $\mathbf{Sign}(\G)=\left\{(3+m^-,m^0,1+m^+),m^-+m^0+m^+=2) \right\}$.
We consider the Euclidean inner product $\prs$ on $L_{6,11}$ such that $\B=(e_6,e_3,e_4,e_1,e_2,e_5)$ is orthogonal with $a_i=\langle e_i,e_i\rangle$. It is obvious that $\B$ is an orthogonal characteristic basis and, according to Lemma \ref{theo1}, the signature of $\prs$ is $(3+m^-,m^0,1+m^+)$ where $(m^-,m^0,m^+)$ is the signature of the characteristic matrix $\mathrm{R}(\ric,\B)$. Now a direct computation using \eqref{riccinilpotent} and \eqref{matrix} gives
$$\mathrm{R}(\ric,\B)= \left( \begin{array}{cc}
\frac{a_3^2 - a_2a_4}{a_1a_2} & 0 \\ 
0 & \frac{a_4^2 - a_3a_6}{a_1a_3}
\end{array} \right).$$ 
If we take $a_1=a_2=a_3=a_4=a_5
=a_6=1$ we get $\mathrm{R}(\ric,\B) = 0$ and we can use  the inverse function theorem trick.
 So, for a suitable choice of $a_i$, $\mathrm{R}(\ric,\B)$ can have all the possible signatures which prove the theorem for $L_{6,11}$.\\

$\bullet$ {\bf The Lie algebra  $L_{5,3}$}. \\

We have $\G=L_{5,3}=\mathrm{span}\{e_1,\ldots,e_5\}$ with the non vanishing Lie brackets
$$[e_{1},e_{2}]= e_{3}, [e_{1},e_{3}]= e_{4}, $$
and $\mathbf{Sign}(\G)=\{(2,1,2),(2,2,1),(3,0,2),(3,1,1),(4,0,1)\}.$ In this case, the parameter $p$ in \eqref{sign} has two values $p=0$ or $1$, so to realize the signatures in $\mathbf{Sign}(\G)$, we will consider two types of Euclidean inner products on $L_{5,3}$. The first ones are  those satisfying $\dim(Z(\G)\cap[\G,\G]^\perp)=1$ and the second ones are those satisfying $\dim(Z(\G)\cap[\G,\G]^\perp)=0$.

 For the first type, consider the Euclidean inner product  $\prs$ on $L_{5,3}$  for which $\B=(e_4,e_3,e_5,e_1,e_2)$ is orthogonal with $a_i=\langle e_i,e_i\rangle$. Then
$\B$  is a characteristic basis for $\prs$ and it is also nice. Then according to Lemma \ref{theo1} the Ricci signature of $\prs$ is  $(2+m^-,1+m^0,1+m^+)$ where $(m^-,m^0,m^+)$ is the signature of $\mathrm{R}(\ric,\B)$. Now a direct computation using \eqref{eqricci} gives $\mathrm{R}(\ric,\B)= (2\ric(e_3,e_3))= \left(\frac{a_3^2-a_2a_4}{a_1a_2}\right)$ and, for suitable values of  the $a_i$, the Ricci signatures of $\prs$ are $(2,1,2)$, $(2,2,1)$ or $(3,1,1)$.

 For the second type, we consider the basis
$\B=(f_1,f_2,f_3,f_4,f_5)=(e_4,e_3,e_5+e_3+e_1,e_1,e_2)$. The non vanishing Lie brackets in this basis are
\[ [f_2,f_3]=-f_1,\;[f_2,f_4]=-f_1,\; [f_3,f_4]=-f_1,\;[f_3,f_5]=f_2,\;[f_4,f_5]=f_2. \]Consider the Euclidean inner product $\prs$ on $L_{5,3}$ for which $\B$ is orthogonal and $a_i=\langle f_i,f_i\rangle$. We have chosen $\B$ and $\prs$ such that $Z(\G)\cap[\G,\G]^\perp=\{0\}$.
Then
$\B$  is a characteristic basis for $\prs$. Then according to Lemma \ref{theo1} the Ricci signature of $\prs$ is  $(3+m^-,m^0,1+m^+)$ where $(m^-,m^0,m^+)$ is the signature of $\mathrm{R}(\ric,\B)$. Here the situation is more complicated than the first case because $\B$ is not a nice basis and the computation of $\mathrm{R}(\ric,\B)$, which is by the way a $(1\times1)$-matrix, is complicated according to formula \eqref{matrix}. We don't need to give the general expression of $\mathrm{R}(\ric,\B)$,  its value when $a_1=a_4=a_5=1$ and $a_3=2$ will suffice to our purpose. We get
$$\mathrm{R}(\ric,\B)=\left(\frac{12a_2^4+6a_2^3+9a_2^2-a_2-3}{4(2a_2^2+a_2+2)}\right).$$
 It is clear that we can choose $a_2$ such that the signature of $\prs$ is   $(3,0,2)$ or $(4,0,1)$.
 This completes the proof for $L_{5,3}$.\\
 
 $\bullet$ {\bf The Lie algebra  $L_{6,3}$}. \\
 
 The treatment is similar to $L_{5,3}$ with a slight difference, the parameter $p$ takes 1 or 2.
 We have $\G=L_{6,3}=\mathrm{span}\{e_1,\ldots,e_6\}$ with the non vanishing Lie brackets
 $$[e_{1},e_{2}]= e_{3}, [e_{1},e_{3}]= e_{4}, $$
 and $\mathbf{Sign}(\G)=\{(2,2,2),(2,3,1),(3,1,2),(3,2,1),(4,1,1)\}.$
 
 For the first type, consider the Euclidean inner product  $\prs$ on $L_{6,3}$  for which $\B=(e_4,e_3,e_5,e_6,e_1,e_2)$ is orthogonal with $a_i=\langle e_i,e_i\rangle$ and  $\dim(Z(\G)\cap[\G,\G]^\perp)=2$. Then
 $\B$  is a characteristic basis for $\prs$ and it is also nice. Then according to Lemma \ref{theo1} the Ricci signature of $\prs$ is  $(2+m^-,2+m^0,1+m^+)$ where $(m^-,m^0,m^+)$ is the signature of $\mathrm{R}(\ric,\B)$. Now a direct computation using \eqref{eqricci} gives $\mathrm{R}(\ric,\B)= (2\ric(e_3,e_3))= \left(\frac{a_3^2-a_2a_4}{a_1a_2}\right)$ and, for suitable values of  the $a_i$, the Ricci signatures of $\prs$ are $(2,2,2)$, $(2,3,1)$ or $(3,2,1)$.
 
 For the second type, we consider the basis
 $\B=(f_1,f_2,f_3,f_4,f_5,f_6)=(e_4,e_3,e_5,e_1,e_2,e_6+e_3+e_1)$. The non vanishing Lie brackets in this basis are
 \[ [f_2,f_4]=-f_1,\;[f_2,f_6]=-f_1,\; [f_4,f_5]=f_2,\;[f_4,f_6]=f_1,\;[f_5,f_6]=-f_2. \]Consider the Euclidean inner product $\prs$ on $L_{6,3}$ for which $\B$ is orthogonal and $a_i=\langle f_i,f_i\rangle$. We have chosen $\B$ and $\prs$ such that $\dim(Z(\G)\cap[\G,\G]^\perp)=1$.
 Then
 $\B$  is a characteristic basis for $\prs$. Then according to Lemma \ref{theo1} the Ricci signature of $\prs$ is  $(3+m^-,1+m^0,1+m^+)$ where $(m^-,m^0,m^+)$ is the signature of $\mathrm{R}(\ric,\B)$. Here the situation is more complicated than the first case because $\B$ is not a nice basis and the computation of $\mathrm{R}(\ric,\B)$, which is by the way a $(1\times1)$-matrix, is complicated according to formula \eqref{matrix}. We don't need to give the general expression of $\mathrm{R}(\ric,\B)$,  its value when $a_1=a_3=a_4=a_5=a_6=1$ will suffice to our purpose. We get
 $$\mathrm{R}(\ric,\B)=\left(\frac{-4a_2^5+2a_2^3+3a_2-2}{1-a_2-2a_2^3}\right).$$
  It is clear that we can choose $a_2$ such that the signature of $\prs$ is   $(3,1,2)$ or $(4,1,1)$.
  This completes the proof for $L_{6,3}$.\\

$\bullet$ {\bf The Lie algebra  $L_{6,5}$}. \\

The treatment is similar to  $L_{5,3}$.
 We have $\G=L_{6,5}=\mathrm{span}\{e_1,\ldots,e_6\}$ with the non vanishing Lie brackets
 $$[e_{1},e_{2}]= e_{3}, [e_{1},e_{3}]= e_{5},[e_{2},e_{4}]= e_{5} $$
 and $\mathbf{Sign}(\G)=\{(3,1,2),(3,2,1),(4,0,2),(4,1,1),(5,0,1)\}.$
 
 For the first type, consider the Euclidean inner product  $\prs$ on $L_{6,5}$  for which $\B=(e_5,e_3,e_6,e_1,e_2,e_4)$ is orthogonal with $a_i=\langle e_i,e_i\rangle$ and $\dim(Z(\G)\cap[\G,\G]^\perp)=1$. Then
 $\B$  is a characteristic basis for $\prs$ and it is also nice. Then according to Lemma \ref{theo1} the Ricci signature of $\prs$ is  $(3+m^-,1+m^0,1+m^+)$ where $(m^-,m^0,m^+)$ is the signature of $\mathrm{R}(\ric,\B)$. Now a direct computation using \eqref{eqricci} gives $\mathrm{R}(\ric,\B)= (2\ric(e_3,e_3))= \left(\frac{a_3^2-a_2a_5}{a_1a_2}\right)$ and, for suitable values of  the $a_i$, the Ricci signatures of $\prs$ are $(3,1,2)$, $(3,2,1)$ or $(4,1,1)$.
 
 For the second type, we consider the basis
 $\B=(f_1,f_2,f_3,f_4,f_5,f_6)=(e_5,e_3,e_1,e_2,e_4,e_6+e_3+e_1)$. The non vanishing Lie brackets in this basis are
 \[ [f_2,f_3]=-f_1,\;[f_2,f_6]=-f_1,\; [f_3,f_4]=f_2,\;[f_3,f_6]=f_1,\;[f_4,f_5]=f_1,\;[f_4,f_6]=-f_2. \]Consider the Euclidean inner product $\prs$ on $L_{6,5}$ for which $\B$ is orthogonal and $a_i=\langle f_i,f_i\rangle$. We have chosen $\B$ and $\prs$ such that $Z(\G)\cap[\G,\G]^\perp=\{0\}$.
 Then
 $\B$  is a characteristic basis for $\prs$. Then according to Lemma \ref{theo1} the Ricci signature of $\prs$ is  $(4+m^-,m^0,1+m^+)$ where $(m^-,m^0,m^+)$ is the signature of $\mathrm{R}(\ric,\B)$. Here the situation is more complicated than the first case because $\B$ is not a nice basis and the computation of $\mathrm{R}(\ric,\B)$, which is by the way a $(1\times1)$-matrix, is complicated according to formula \eqref{matrix}. We don't need to give the general expression of $\mathrm{R}(\ric,\B)$,  its value when $a_1=a_3=a_4=a_5=a_6=1$  will suffice to our purpose. We get
 $$\mathrm{R}(\ric,\B)=
 \left(\frac{4a_2^6+6a_2^5+6a_2^4-a_2^3-3a_2^2-3a_2-1}{a_2(2a_2^3+3a_2^3+2a_2+1)}\right).$$
  It is clear that we can choose $a_2>0$ such that the signature of $\prs$ is   $(4,0,2)$ or $(5,0,1)$.
  This completes the proof for $L_{6,5}$.\\
  
 $\bullet$ {\bf The Lie algebra  $L_{6,9}$}. \\
  
  The treatment is similar to $L_{5,3}$ and $L_{6,5}$.
   We have $\G=L_{6,9}=\mathrm{span}\{e_1,\ldots,e_6\}$ with the non vanishing Lie brackets
   $$[e_{1},e_{2}]= e_{3}, [e_{1},e_{3}]= e_{4},[e_{2},e_{3}]= e_{5}, $$
    and $\mathbf{Sign}(\G)=\{(2,1,3),(2,2,2),(3,0,3),(3,1,2),(4,0,2)\}.$
   
   For the first type, consider the Euclidean inner product  $\prs$ on $L_{6,9}$  for which $\B=(e_5,e_4,e_3,e_6,e_1,e_2)$ is orthogonal with $a_i=\langle e_i,e_i\rangle$ and  $\dim(Z(\G)\cap[\G,\G]^\perp)=1$. Then
   $\B$  is a characteristic basis for $\prs$ and it is also nice. Then according to Lemma \ref{theo1} the Ricci signature of $\prs$ is  $(2+m^-,1+m^0,2+m^+)$ where $(m^-,m^0,m^+)$ is the signature of $\mathrm{R}(\ric,\B)$. Now a direct computation using \eqref{eqricci} gives $\mathrm{R}(\ric,\B)= (2\ric(e_3,e_3))= \left(\frac{a_3^2-a_2(a_4+a_5)}{a_1a_2}\right)$ and, for suitable values of  the $a_i$, the Ricci signatures of $\prs$ are $(2,1,3)$, $(2,2,2)$ or $(3,1,2)$.
   
   For the second type, we consider the basis
   $\B=(f_1,f_2,f_3,f_4,f_5,f_6)=(e_5,e_4,e_3,e_1,e_2,e_6+e_3+e_1)$. The non vanishing Lie brackets in this basis are
   \[ [f_3,f_4]=-f_2,\;[f_3,f_5]=-f_1,\; [f_3,f_6]=-f_2,\;[f_4,f_5]=f_3,\;[f_4,f_6]=f_2,\;[f_5,f_6]=f_1-f_3. \]Consider the Euclidean inner product $\prs$ on $L_{6,9}$ for which $\B$ is orthogonal and $a_i=\langle f_i,f_i\rangle$. We have chosen $\B$ and $\prs$ such that $Z(\G)\cap[\G,\G]^\perp=\{0\}$.
   Then
   $\B$  is a characteristic basis for $\prs$. Then according to Lemma \ref{theo1} the Ricci signature of $\prs$ is  $(3+m^-,m^0,2+m^+)$ where $(m^-,m^0,m^+)$ is the signature of $\mathrm{R}(\ric,\B)$. Here the situation is more complicated than the first case because $\B$ is not a nice basis and the computation of $\mathrm{R}(\ric,\B)$, which is by the way a $(1\times1)$-matrix, is complicated according to formula \eqref{matrix}. We don't need to give the general expression of $\mathrm{R}(\ric,\B)$,  its value when $a_1=a_2=a_3=a_5=a_6=1$ will suffice to our purpose. We get
   $$\mathrm{R}(\ric,\B)=\left(\frac{12-a_4-35a_4^2}{2(8a_4+3)a_4}\right).$$
    It is clear that we can choose $a_4$ such that the signature of $\prs$ is   $(3,0,3)$ or $(4,0,2)$.
    This completes the proof for $L_{6,9}$.\\

 $\bullet$ {\bf The Lie algebra  $L_{6,6}$}. \\
 
 The situation here is different from the precedent cases. We still have two types of Euclidean products but the order of the  reduced matrix of the Ricci curvature is 2.
 We have $\G=L_{6,6}=\mathrm{span}\{e_1,\ldots,e_6\}$ with the non vanishing Lie brackets
 $$[e_{1},e_{2}]= e_{3}, [e_{1},e_{3}]= e_{4}, [e_{1},e_{4}]= e_{5}, [e_{2},e_{3}]= e_{5} $$
 and $\mathbf{Sign}(\G)=\left\{(2,1,3),(2,2,2),(2,3,1),(3,0,3),(3,1,2),(3,2,1),(4,0,2),(4,1,1),(5,0,1)\right\}.$
 
 For the first type, consider the Euclidean inner product  $\prs$ on $L_{6,6}$  for which $\B=(e_5,e_3,e_4,e_6,e_1,e_2)$ is orthogonal with $a_i=\langle e_i,e_i\rangle$ and  $\dim(Z(\G)\cap[\G,\G]^\perp)=1$. Then
   $\B$  is a characteristic basis for $\prs$ and it is also nice. Then according to Lemma \ref{theo1} the Ricci signature of $\prs$ is  $(2+m^-,1+m^0,1+m^+)$ where $(m^-,m^0,m^+)$ is the signature of $\mathrm{R}(\ric,\B) =\mathrm{diag}(2\ric(e_3,e_3),2\ric(e_4,e_4)) $. Now a direct computation using \eqref{riccinilpotent} gives 
   \[2\ric(e_3,e_3)= \frac{a_3^2-a_2a_4-a_1a_5}{a_1a_2}\esp2\ric(e_4,e_4)= \frac{a_4^2-a_3a_5}{a_1a_3}.\] 
    If we take $a_1=6,a_2=5,a_3=4,a_4=2,a_5
    =a_6=1$, we get $\mathrm{R}(\ric,\B) = 0$ and we can apply the inversion theorem trick to get that for a suitable choice of the $a_i$ the Ricci signature of $\prs$ is $(2,1,3)$, $(2,2,2)$ ,$(2,3,1)$, $(3,1,2)$, $(3,2,1)$ or $(4,1,1)$.

    For the second type, we consider the basis $\B=(f_1,f_2,f_3,f_4,f_5,f_6)$ and the Euclidean inner product $\prs$ on $L_{6,6}$ for which $\B$ is orthogonal and $a_i=\langle f_i,f_i\rangle$. We choose $\B$ and $\prs$ such that $Z(\G)\cap[\G,\G]^\perp=\{0\}$.
  \begin{itemize}

 \item    $\B=(f_1,f_2,f_3,f_4,f_5,f_6)=(e_5,e_3,e_4,e_1,e_2,e_6+e_3)$. The non vanishing Lie brackets in this basis are
      \[ [f_2,f_4]=-f_3,\;[f_2,f_5]=-f_1,\; [f_3,f_4]=-f_1,\;[f_4,f_5]=f_2,\;[f_4,f_6]=f_3,\;[f_5,f_6]=f_1. \]
      Then
      $\B$  is a characteristic basis for $\prs$ and is not nice. Then, according to Lemma \ref{theo1}, the Ricci signature of $\prs$ is $(3+m^-,m^0,1+m^+)$ where $(m^-,m^0,m^+)$ is the signature of $\mathrm{R}(\ric,\B)$.  Now, a direct computation using \eqref{riccinilpotent} and \eqref{matrix} gives
      $$\mathrm{R}(\ric,\B) = \mathrm{diag}\left(\frac{a_2^2}{a_4 a_5},\frac{-a_1 a_2 a_6+a_3^2 \left(a_2+a_6\right)}{a_2 a_4 a_6}\right).$$Thus, for suitable values of $a_i$, the signatures  $(3,0,3)$ and $(4,0,2)$ are realizable as the Ricci signature of $\prs$.

  \item 
 $\B=(f_1,f_2,f_3,f_4,f_5,f_6)=(e_5,e_3,e_4,e_1,e_2,e_6+e_3+e_1)$. The non vanishing Lie brackets are
 \[ [f_2,f_4]=-f_3,\;[f_2,f_5]=-f_1,\; [f_2,f_6]=-f_3,\;[f_3,f_4]=-f_1,\;[f_3,f_6]=-f_1,[f_4,f_5]=f_2,\;[f_4,f_6]=f_3,\]\[[f_5,f_6]=-f_2+f_1. \] Then
 $\B$  is a characteristic for $\prs$. According to Lemma \ref{theo1}, the Ricci signature of $\prs$ is  $(3+m^-,m^0,1+m^+)$ where $(m^-,m^0,m^+)$ is the signature of $\mathrm{R}(\ric,\B)$.
  Here the situation is more complicated than the first case because $\B$ is not a nice basis and the computation of $\mathrm{R}(\ric,\B)$, which is by the way a $(2\times2)$-matrix, is complicated according to formula \eqref{matrix}. We don't need to give the general expression of $\mathrm{R}(\ric,\B)$ ,  its value when $a_1=3,a_2=a_4=a_5=2=a_6 = 1$ will suffice to our purpose. We get
   $$\mathrm{R}(\ric,\B) = \mathrm{diag}\left(
   -\frac{18+66 a_3+121 a_3^2+120 a_3^3+73 a_3^4+24 a_3^5}{18+36 a_3+34 a_3^2+22 a_3^3+6 a_3^4},\frac{-57+8 a_3^2}{8}\right)$$
    It is clear that for suitable values of $a_3$, the signature $(5,0,1)$ is realizable as the Ricci signature of $\prs$.
    This completes the proof for $L_{6,6}$\\
 \end{itemize}

 $\bullet$ {\bf The Lie algebra  $L_{6,7}$}. \\
 
 The treatment is similar to $L_{6,6}$.
 We have $\G=L_{6,7}=\mathrm{span}\{e_1,\ldots,e_5\}$ with the non vanishing Lie brackets
 $$[e_{1},e_{2}]= e_{3}, [e_{1},e_{3}]= e_{4}, [e_{1},e_{4}]= e_{5} $$
 and $\mathbf{Sign}(\G)=\{(2,1,3),(2,2,2),(2,3,1),(3,0,3),(3,1,2),(3,2,1),(4,0,2),(4,1,1),(5,0,1)\}.$
 
 For the first type, consider the Euclidean inner product  $\prs$ on $L_{6,7}$  for which $\B=(e_5,e_3,e_4,e_6,e_1,e_2)$ is orthogonal with $a_i=\langle e_i,e_i\rangle$ and  $\dim(Z(\G)\cap[\G,\G]^\perp)=1$. Then
   $\B$  is a characteristic basis for $\prs$ and it is also nice. Then according to Lemma \ref{theo1} the Ricci signature of $\prs$ is  $(2+m^-,1+m^0,1+m^+)$ where $(m^-,m^0,m^+)$ is the signature of $\mathrm{R}(\ric,\B) =\mathrm{diag}(2\ric(e_3,e_3),2\ric(e_4,e_4)) $. Now a direct computation using \eqref{riccinilpotent} gives 
   \[2\ric(e_3,e_3)= \frac{a_3^2-a_2a_4}{a_1a_2}\esp2\ric(e_4,e_4)= \frac{a_4^2-a_3a_5}{a_1a_3}.\] 
    If we take $a_1=a_2=a_3=a_4=a_5
    =a_6=1$ we get $\mathrm{R}(\ric,\B) = 0$ and we can apply the inversion theorem trick to get that for a suitable choice of the $a_i$ the Ricci signature of $\prs$ is $(2,1,3)$, $(2,2,2)$ ,$(2,3,1)$, $(3,1,2)$, $(3,2,1)$ or $(4,1,1)$.

 For the second type, we consider the basis $\B=(f_1,f_2,f_3,f_4,f_5,f_6)$ and the Euclidean inner product $\prs$ on $L_{6,7}$ for which $\B$ is orthogonal and $a_i=\langle f_i,f_i\rangle$. We choose $\B$ and $\prs$ such that $Z(\G)\cap[\G,\G]^\perp=\{0\}$.
 \begin{itemize}
 
 \item 
      $\B=(f_1,f_2,f_3,f_4,f_5,f_6)=(e_5,e_3,e_4,e_1,e_2,e_6+e_3)$. The non vanishing Lie brackets in this basis are
      \[ [f_2,f_4]=-f_3,\; [f_3,f_4]=-f_1,\;[f_4,f_5]=f_2,\;[f_4,f_6]=f_3. \]   Then
      $\B$  is a characteristic basis for $\prs$ and is not nice. Then according to Lemma \ref{theo1} the Ricci signature of $\prs$ is $(3+m^-,m^0,1+m^+)$ where $(m^-,m^0,m^+)$ is the signature of $\mathrm{R}(\ric,\B)$. Now a direct computation using \eqref{riccinilpotent} and \eqref{matrix} gives
      $$\mathrm{R}(\ric,\B) = \mathrm{diag}\left(\frac{a_2^2}{a_4 a_5},\frac{a_2 a_3^2+\left(-a_1 a_2+a_3^2\right) a_6}{a_2 a_4 a_6}\right).$$Thus for suitable values of $a_i$, the signatures  $(3,0,3)$ and $(4,0,2)$ are realizable as the Ricci signature of $\prs$.

 \item 
 $\B=(f_1,f_2,f_3,f_4,f_5,f_6)=(e_5,e_3,e_4,e_1,e_2,e_6+e_3+e_1)$. The non vanishing brackets are
 \[ [f_2,f_4]=-f_3,\;[f_2,f_6]=-f_3,\; [f_3,f_4]=-f_1,\;[f_3,f_6]=-f_1,\;[f_4,f_5]=f_2 ,[f_4,f_6]=f_3,\;[f_5,f_6]=-f_2 \] Then
 $\B$  is a characteristic for $\prs$. Then according to Lemma \ref{theo1} the Ricci signature of $\prs$ is  $(3+m^-,m^0,1+m^+)$ where $(m^-,m^0,m^+)$ is the signature of $\mathrm{R}(\ric,\B)$.
 Here the situation is more complicated than the first case because $\B$ is not a nice basis and the computation of $\mathrm{R}(\ric,\B)$, which is by the way a $(2\times2)$-matrix, is complicated according to formula \eqref{matrix}. We don't need to give the general expression of $\mathrm{R}(\ric,\B)$,  its value when $a_1= 2,a_2=a_3=a_4=a_6=1$  will suffice to our purpose. We get
   $$\mathrm{R}(\ric,\B) = \mathrm{diag}\left(\frac{8+17a_5-12a_5^2}{4(1+3a_5)a_5},-2\right)$$
    It is clear that we can choose $a_5$ such that the signature of $\prs$ is $(5,0,1)$.
    This completes the proof for $L_{6,7}$
 \end{itemize}
 
\end{proof}

We end this work by giving all the realizable Ricci signatures on nilpotent Lie groups up to dimension 6.

\begin{center}
 \begin{tabular}{|*{2}{c|}}
  \hline
  Lie algebra $\G$ & Realizable Ricci signatures\\
  \hline
  $ L_{3,2}$& $(2,0,1)$\\
  \hline
  $L_{4,2}$&$(2,1,1)$ \\
  \hline
  $L_{4,3}$&$(2,1,1),(2,0,2),(3,0,1)$ \\
  \hline
  $L_{5,2}$& $(2,2,1)$\\
  \hline
  $ L_{5,3}$& $(2,1,2),(2,2,1), (3,0,2), (3,1,1),
(4,0,1)$ \\
  \hline
 $ L_{5,4}$ & $(4,0,1)$ \\
  \hline
  $L_{5,5}$ & $(3,0,2),(3,1,1),(4,0,1)$\\
  \hline
  $L_{5,6},L_{5,7}$ & $(2,0,3),(2,1,2),(2,2,1),(3,0,2),
(3,1,1),(4,0,1)$\\
    \hline
$L_{5,8}$ & $(3,0,2)$ \\
  \hline
  $L_{5,9}$ & $(2,0,3),(2,1,2),(3,0,2)$\\
  \hline
  $L_{6,2}$& $(2,3,1)$\\
  \hline
  $ L_{6,3}$& $(2,2,2),(2,3,1),(3,1,2),(3,2,1),(4,1,1) $\\
  \hline
 $ L_{6,4}$ & $(4,1,1)$ \\
  \hline
  $L_{6,5}$ & $(3,1,2),(3,2,1),(4,0,2),(4,1,1),(5,0,1)$\\
  \hline
  $L_{6,6},L_{6,7}$ & $(2,1,3),(2,2,2),(2,3,1),(3,0,3),(3,1,2)$,\\ &$(3,2,1),(4,0,2), (4,1,1),(5,0,1)$\\
    \hline
$L_{6,8}$ & $(3,1,2)$ \\
  \hline
  $L_{6,9}$ & $(2,1,3),(2,2,2),(3,0,3),(3,1,2),(4,0,2)$\\
  \hline
  $L_{6,10}$&$(4,0,2),(4,1,1),(5,0,1)$\\
  \hline
  $L_{6,11},L_{6,12},L_{6,13},L_{6,20}$&$(3,0,3),(3,1,2),(3,2,1),$\\
  $L_{6,19}(\epsilon),\epsilon\in\{-1,1\}$ &$(4,0,2),(4,1,1),(5,0,1)$\\
    \hline
 $L_{6,14},L_{6,15},L_{6,16},L_{6,17}$&$(2,0,4),(2,1,3),(2,2,2),(2,3,1),(3,0,3)$,\\
 $L_{6,18},L_{6,21}(\epsilon),\epsilon\in\{-1,1\}$ &$(3,1,2),(3,2,1),(4,0,2
  ),(4,1,1),(5,0,1)$\\
  \hline
  $L_{6,19}(0),L_{6,23},L_{6,25}$&$(3,0,3),(3,1,2),(4,0,2) $\\
$ L_{6,24}(\epsilon),\epsilon\in\{-1,0,1\}$ & \\
  \hline
  $L_{6,21}(0)$&$(2,0,4),(2,1,3),(2,2,2),$\\ &$(3,0,3),(3,1,2),(4,0,2)$\\
  \hline
  $L_{6,22}(\epsilon),\epsilon\in\{-1,0,1\}$&$(4,0,2)$\\
     \hline
  $L_{6,26}$&$(3,0,3)$\\
  \hline
  \end{tabular}
  \end{center}
  \begin{center}
  {\bf Table 3}: Realizable Ricci signatures on nilpotent Lie groups of dimension $\leq 6$.
  \end{center}

%%%%%%%%%%%%%%%%%%%%%%%%%%%%%%%%%%%%%%%%%%%%%%%%%%%%%%%%%%%%%%%%%%%%%%%%%%%

\end{document}